\documentclass[a4paper]{article}

\usepackage[utf8x]{inputenc}
\usepackage{latexsym, amssymb, amsmath, amsthm, amsfonts, mathrsfs}

\usepackage{cite}

\usepackage{url}
\makeatletter
\def\url@leostyle{%
  \@ifundefined{selectfont}{\def\UrlFont{\sf}}{\def\UrlFont{\small\ttfamily}}}
\makeatother
\theoremstyle{plain}

\newtheorem{thm}{Theorem}

\newtheorem{lem}{Lemma}
\newtheorem{cor}{Corollary}

\theoremstyle{remark}

\theoremstyle{definition}
\newtheorem{dfi}{Definition}

\newcommand{\scalpr}[2]{\langle #1 , #2 \rangle}
\renewcommand{\det}{\operatorname{det}\,}
\renewcommand{\div}{\operatorname{div}}

\newcommand{\grad}{\operatorname{grad}}

\newcommand{\tr}{\operatorname{tr}}

\newcommand{\N}{\mathbb{N}}

\newcommand{\R}{\mathbb{R}}
\newcommand{\Ss}{\mathbb{S}}
\newcommand{\Z}{\mathbb{Z}}
\renewcommand{\cal}{\mathcal}

\begin{document}

\author{
{\large Ernst Kuwert and Johannes Lorenz}
\footnote{The authors acknowledge support by Deutsche Forschungsgemeinschaft
via SFB/Transregio 71 {\em Geometric Partial Differential Equations}.}
}

\title{On the Stability of the CMC Clifford Tori as Constrained Willmore Surfaces}

\maketitle

\begin{abstract}
The tori $T_r = r\, {\mathbb S}^1 \times s \, {\mathbb S}^1 \subset {\mathbb S}^3$,
where $r^2 + s^2 = 1$, are constrained Willmore surfaces, i.e. critical points of 
the Willmore functional among tori of the same conformal type. We compute which
of the $T_r$ are stable critical points.
\end{abstract}
%\vspace{0.5cm}
%{\bf Keywords} Willmore surfaces, Stability, Clifford Torus, CMC surfaces \\
%{\bf Mathematics Subject Classification (2000)} 53A30, 53C21, 49Q10
%%%%%%%%%%%%%%%%%%%%%%%%%%
%   INTRODUCTION         %
%%%%%%%%%%%%%%%%%%%%%%%%%%

\section*{Introduction}
For an immersed closed surface $f: \Sigma \rightarrow {\mathbb S}^3$ the Willmore functional is
$$
\cal W(f) = \int _{\Sigma} \left( \frac{1}{4}  |\vec H|^2 + 1 \right) \, d \mu_g,
$$
where $\vec H$ is the mean curvature vector in ${\mathbb S}^3$ and $g$ is the induced metric.
Critical points are called Willmore surfaces. They are characterized by the 
Euler-Lagrange equation
\begin{equation*}
\vec{W}(f) = \Delta ^{\perp} \vec H + g^{ij} g^{kl} A^\circ _{ik} \scalpr{A _{jl}^\circ}{\vec H} = 0.
\end{equation*}
Here $\Delta^{\perp}$ denotes the Laplacian in the normal bundle along $f$, and $A^\circ$ 
is the tracefree component of the vector-valued second fundamental form $A$. By definition,
$f:\Sigma \to {\mathbb S}^3$ is a (conformally) constrained Willmore surface if it is a critical point of 
${\cal W}$ with respect to variations in the class of surfaces having the same conformal 
type. In other words, if $\pi:{\cal M}(\Sigma) \to {\cal T}(\Sigma)$ denotes the projection 
from Riemannian metrics onto Teichm\"uller space, then the point $\pi(f^\ast g_{{\mathbb S}^3})$ 
is prescribed \cite{Tro,fistro,KS}. The resulting Euler-Lagrange equation is
\begin{equation}
\label{constwill}
\vec{W}(f) = g(A^\circ,q) \quad \mbox{ for some } q \in S^{TT}_2 (g).
\end{equation}
In the case of a torus, the space $S^{TT}_2(g)$ of symmetric, covariant $2$-tensors $q$ 
with ${\rm div}_g q = 0$ and ${\rm tr}_g(q) = 0$ (transverse traceless) is two-dimensional.
For immersions at which the projection onto the Teichm\"{u}ller space has full rank, 
the tensor $q \in S^{TT}_2 (g)$ in equation (\ref{constwill}) is obtained from 
the Lagrange multiplier rule. A loss of rank occurs precisely when there is 
a nonzero $q \in S^{TT}_2 (g)$ such that $g(A^\circ,q) \equiv 0$, that is 
the surface is isothermic \cite{BPP,KS}. For example, rotationally symmetric 
surfaces and also constant mean curvature (CMC) surfaces are isothermic. 
Recently, Sch\"atzle established the Euler-Lagrange equation (\ref{constwill}) 
also in that degenerate case \cite{Sch}. Reversely, a solution to (\ref{constwill}) 
is always a constrained Willmore surface.\\
\\ 
In this paper we study the CMC Clifford tori 
$T_r = r\, {\mathbb S}^1 \times s\,{\mathbb S}^1 \subset {\mathbb S}^3$, 
assuming always $r^2 + s^2 = 1$, which are isometrically parametrized by 
$$
f_r:\Sigma_r = {\mathbb R}^2 /2\pi r {\mathbb Z} \times 2\pi s {\mathbb Z} \to {\mathbb S}^3,\,
f_r(u,v) = \left(\begin{array}{c} 
r\, e^{iu/r}\\ 
s\, e^{iv/s} 
\end{array} \right).
$$
The $f_r$ are isothermic, in fact we have $g(A^\circ,q_1) \equiv 0$
for $q_1 = du \otimes dv + dv \otimes du$. Moreover, they are 
constrained Willmore surfaces since 
\begin{equation}
\label{constwillclifford}
\vec{W}(f) = \frac{r^2-s^2}{2r^2 s^2} g(A^\circ,q_2) \quad \mbox{ where }\, 
q_2 = dv \otimes dv - du \otimes du.
\end{equation}
The following answers the question
for which parameters $r \in (0,1)$ the $T_r$ are stable critical 
points of the Willmore functional in the class of surfaces
having the same conformal type. 

\begin{thm}
The tori $T_r = r\,{\mathbb S}^1 \times s \, {\mathbb S}^1 \subset {\mathbb S}^3$, 
$r^2 + s^2 = 1$, are stable constrained Willmore surfaces if and only if 
$$
r \in \Big[\frac{1}{2},\frac{\sqrt{3}}{2}\Big] \quad \mbox{ or equivalently } \quad
b \in \Big[\frac{1}{\sqrt{3}},\sqrt{3}\Big]. 
$$
Here $a+ib \in {\mathbb H}$ are standard coordinates on the Teichm\"uller 
space of the torus. 
\end{thm}

The stability for $\frac{1}{2} \leq r \leq \frac{\sqrt{3}}{2}$ is proved
under the weaker condition that only the $b$-coordinate in Teichm\"uller 
space is prescribed. This is in line with a recent result of Ndiaye and 
Sch\"atzle \cite{NSch}. For $r$ sufficiently close to $\frac{1}{\sqrt{2}}$,
they prove that the $T_r$ actually minimize within the class of surfaces
having the same coordinate $b$ in Teichm\"uller space. For the unstable
case, we show that the tori for $r < \frac{1}{2}$ or $r > \frac{\sqrt{3}}{2}$
are unstable already in the class of rotationally symmetric surfaces.
For $\frac{1}{k+2} \leq r < \frac{1}{k+1}$ the stability operator
has exactly $k$ negative eigenvalues for $k = 1,2,\ldots$. 
Bifurcations of the tori $T_r$ as CMC surfaces are studied in 
\cite{AP10}. In \cite{KSS} equivariant CMC tori are computed 
using spectral curve methods. 

%%%%%%%%%%%%%%%%%%%%%%%%%%
%   DEFINITIONS          %
%%%%%%%%%%%%%%%%%%%%%%%%%%

\section{Definitions}
Here we collect the basic definitions regarding stability.
We denote by ${\rm Imm}(\Sigma,{\mathbb S}^3)$ the space of
immersions of a closed surface $\Sigma$ into the $3$-sphere.
In applications of the implicit function theorem and also
in Teichm\"uller theorem one has to specify an appropriate
degree of smoothness for the surfaces, however this is omitted
for the sake of presentation. We assume that the functionals are 
twice continuously differentiable. 

\begin{dfi} 
\label{defcritical}
Let ${\cal G}:{\rm Imm}(\Sigma, {\mathbb S}^3) \rightarrow \R^k$, 
$f_0 \in {\rm Imm}(\Sigma,{\mathbb S }^3)$ be given, and put
$z_0 = {\cal G}(f_0)$. We say that $f_0$ is a critical point 
for the Willmore functional under the constraint ${\cal G}$, 
if and only if 
$$
\frac{d}{dt} {\mathcal W}(f(\cdot,t))\Big|_{t=0} = 0 \quad 
\mbox{ for all admissible variations } f(\cdot,t),
$$
that is $f(\cdot,0) = f_0$ and ${\cal G}(f(\cdot,t)) = z_0$ for all $t$.
\end{dfi}

If $\vec{W}(f_0)$ is $L^2$-orthogonal to the kernel of $D{\cal G}(f_0)$ with 
respect to the metric induced by $f_0$, then $f_0$ is a critical point 
under the constraint ${\cal G}$. The reverse implication follows from
the implicit function theorem if the differential $D{\cal G}(f_0)$ 
is surjective. 

\begin{dfi} [Constrained Willmore]
An immersed surface $f:\Sigma \rightarrow {\mathbb S}^3$  is called 
constrained Willmore if it is a critical point of $\mathcal W$ under
the constraint given by projecting onto the Teichm\"uller space.
\end{dfi}

We recall that the Teichm\"uller space is a finite-dimensional 
manifold, so that we have a $\R^k$-valued constraint by chosing a 
chart. It is easy to see that the space $Q= \{g(A^\circ,q): q \in S^{TT}_2(g)\}$ 
is the $L^2$-orthogonal complement of the kernel of the linearized projection 
at $f_0$. In particular, the equation $\vec{W}(f) = g(A^\circ,q)$ for some 
$q \in S^{TT}_2(g)$ implies that $f$ is constrained Willmore. The reverse
is clear if $f$ is not isothermic. However the reverse also holds in 
the degenerate case when $f$ is isothermic, as proved in \cite{KS,Sch}.

\begin{dfi} [Stability]
Let $f_0:\Sigma \rightarrow {\mathbb S}^3$ be critical for the Willmore 
functional under the constraint ${\cal G}: {\rm Imm}(\Sigma, {\mathbb S}^3) \to \R^k$. 
Then $f_0$ is called stable (under the constraint ${\cal G}$) if and only if 
$$
\frac{d^2}{dt^2} \cal W (f_t) \Big|_{t=0} \geq 0
$$
for all admissible variations (see Definition \ref{defcritical}).
\end{dfi}

In the nondegenerate case when ${\cal G}$ has full rank at $f_0$ we can 
give an infinitesimal characterization of stability. 

\begin{lem} \label{lemstability}
Let $f_0$ be critical for the Willmore functional under the constraint 
${\cal G}: {\rm Imm}(\Sigma, {\mathbb S}^3) \rightarrow \R^k$, and suppose 
that ${\cal G}$ has full rank at $f_0$, i.e. ${\rm rank\,}D{\cal G}(f_0) = k$. 
Then for any admissible variation $f(\cdot,t)$ with $f(\cdot,0) = f_0$ and 
${\cal G}(f(\cdot,t)) = z_0$ for all $t$ we have the formula
\begin{equation}
\label{secondvariation1}
\frac{d^2}{dt^2} \cal W (f_t)|_{t=0}
= D^2 \cal W (f_0) (\phi, \phi) - \sum _{i=1} ^k \lambda _i D^2 {\cal G}_i(f_0) (\phi,\phi),
\end{equation}
where $\phi = \frac{\partial f}{\partial t}(\cdot,0)$.
Here the $\lambda_i \in \R$ are Lagrange multipliers given by
\begin{equation}
\label{secondvariation2}
D\cal W (f_0) = \sum _{i=1} ^k \lambda_i D{\cal G}_i(f_0). 
\end{equation}
In particular, $f_0$ is stable under the constraint ${\cal G}$ if and only if 
the quadratic form on the right of (\ref{secondvariation1}) is positive semidefinite. 
\end{lem}

\begin{proof}
Using the covariant derivative in ${\mathbb S}^3$, we can write 
$$
\frac{d^2}{dt^2}\cal W(f_t) |_{t=0} = D^2 \cal W (f_0) (\phi, \phi) + D\cal W (f_0) (\psi),
\quad \mbox{ where } \psi = \frac{D}{\partial t} \frac{\partial f}{\partial t}(\cdot,0).
$$
By assumption, there exist vectorfields $\phi_i$ along $f_0$ with
$$
D{\cal G}(f_0) \phi_i = e_i \quad \mbox{ for } i = 1,\ldots,k.
$$
Since $f_0$ is critical under the constraint ${\cal G}$, there are 
$\lambda_1,\ldots,\lambda_k \in \R$ with 
$$
D\cal W (f_0) \phi = \sum _{i=1} ^k \lambda_i D{\cal G}_i(f_0) \phi 
\quad \mbox{ for all }\phi.
$$
Applying to $\phi_j$ for $j = 1,\ldots,k$ and using that $D{\cal G}_i(f_0)\phi_j = \delta_{ij}$, 
we see that
$$
\lambda_i = D\cal W (f_0) \phi_i.
$$
With that we calculate
\begin{eqnarray*}
D\cal W(f_0) (\psi) & = & \sum _{i=1}^k \lambda_i D{\cal G}_i(f_0) \psi \\
& = & \sum _{i=1} ^k \lambda _i 
\Big(\underbrace{\frac{d^2}{dt^2} {\cal G}_i (f_t)\Big|_{t=0}}_{ = 0} - D^2 {\cal G}_i (\phi,\phi)\Big) \\
& = & - \sum_{i=1}^k \lambda _i  D^2 {\cal G}_i (\phi,\phi).
\end{eqnarray*}
Plugging in completes the proof.
\end{proof}

%%%%%%%%%%%%%%%%%%%
%  Stability      %
%%%%%%%%%%%%%%%%%%%

\section{Stability of the tori $T_r$} 
For $r^2+s^2 = 1$, let $\Sigma_r = \mathbb{R}^{2}/(2 \pi r  \mathbb{Z} \times 2 \pi s \mathbb{Z})$
and consider the embedded tori 
$$
f_{r}: \Sigma_r \to {\mathbb S}^3,\, f_r(u,v) =
\begin{pmatrix} 
 r\,e^{iu/r} \\  
 s\, e^{iv/s} 
\end{pmatrix}. 
$$
We will calculate the basic geometric data for the $f_r$. We have 
$$
\partial_{1} f _r = \begin{pmatrix} i e^{iu/r} \\ 0 \end{pmatrix} \quad 
\partial_{2} f _r = \begin{pmatrix} 0 \\ i e^{iv/s} \end{pmatrix}, 
$$
and in particular 
$$
g _{ij} \equiv \delta _{ij}. 
$$
The unit normal along $f_r$ in ${\mathbb S}^3$ is given by
$$
\vec{n} = 
\begin{pmatrix}  
s e^{iu/r}\\ 
- r e^{iv/s}
\end{pmatrix}.
$$
For the second fundamental form we get 
\begin{equation}
A _{11} = - \frac{s}{r} \vec{n},\,A_{22}  = \frac{r}{s} \vec{n},\, 
A_{12}  = A_{21} = 0.  
\end{equation}
The mean curvature vector is given by 
\begin{equation} \label{eqmeancurv}
\vec{H} = A_{11} + A_{22} = \frac{r^2-s^2} {rs} \vec{n},
\end{equation}
and the tracefree part of the second fundamental form is
\begin{equation}
A^{\circ}_{11} = - \frac{1}{2rs} \vec{n} = - A^{\circ}_{22},\,
A^\circ_{12} = A^\circ_{21} = 0. 
\end{equation}
Using $\nabla \vec{n} = 0$ and $g_{ij} = \delta_{ij}$, we further compute 
\begin{equation}
\label{eqfirstvariation}
\vec{W} (f_{r}) =
%|A^\circ|^2 \vec{H} = 
\frac{r^2 - s^2}{2 r^3 s^3} \vec{n}.
\end{equation}
Now define the tensors $q_i \in S^{TT}_2(g)$, $i = 1,2$, by
$$
q_1 = du \otimes dv + dv \otimes du,  \quad q_2 = dv \otimes dv - du \otimes du.
$$
As $g(A^\circ,q_1) \equiv 0$ we see that $f_r$ isothermic. 
Moreover $f_r$ is constrained Willmore, namely we easily compute 
$$
\vec{W} (f_r) = \frac{r^2-s^2}{2r^2 s^2} g(A^\circ,q_2). 
$$

To fix a parametrization of ${\cal T}(\Sigma_r)$, we consider the linear maps
$$
L_{a,b} = 
\left(\begin{array}{cc} 1 & a \\ 0 & b \end{array}\right)\,
\left(\begin{array}{cc} \frac{1}{2\pi r} & 0 \\ 0 & \frac{1}{2\pi s}\end{array} \right). 
$$
These induce an isomorphism between $\Sigma_r$ and 
$T_{a,b} = \R^2/\big({\mathbb Z} \oplus {\mathbb Z} (a,b)\big)$, 
with 
$$
(2\pi r)^2 L_{a,b}^\ast (g_{euc}) =
du^2 + 2\,a \,\frac{r}{s} du dv + (a^2+b^2)\,\frac{r^2}{s^2}\, dv^2 
= : g_{a,b}.
$$
By Teichm\"uller theory on the torus, see \cite{Tro}, the map
$$
\varphi_r:{\mathbb H} \to {\cal T}(\Sigma_r),\, \varphi_r(a,b) = \pi_{\Sigma_r} (g_{a,b}),
$$
is a diffeomorphism, hence we may equivalently consider the projection 
$$
\pi = \varphi_r^{-1} \circ \pi_{\Sigma_r}:{\cal M}(\Sigma_r) \to {\mathbb H},
$$
in particular we have 
$$
\pi(g_{a,b}) = (a,b) \in {\mathbb H}.
$$
Note that $g_{a,b} = g_{euc}$ for $a = 0$, $b = s/r$. Taking the derivative yields 
$$
D\pi(g_{euc}) q^\mu = \frac{s}{r} e_\mu \quad \mbox{ for } \mu = 1,2.
$$
Now introduce the map 
$$
G: {\rm Imm}(\Sigma_r,{\mathbb S}^3) \to {\cal M}(\Sigma_r),\,
G(f) = f^\ast g_{{\mathbb S}^3},
$$
as well as the compositions 
$$
\Pi = ({\cal A}, {\cal B}) = \pi \circ G: 
{\rm Imm}(\Sigma_r,{\mathbb S}^3) \to {\mathbb H}.
$$

\begin{lem} Let $(M,h)$ be a Riemannian manifold. For any closed immersed 
surface $f \in {\rm Imm}(\Sigma,M)$ we let $G(f) = f^\ast(h)$. Then we 
have the formula 
$$
\langle DG(f) \phi,q \rangle_{L^2(g)} = - 2 \int_\Sigma \big\langle g(A^\circ,q),\phi \big\rangle\,d\mu_g
$$
where $g = f^\ast(h)$ and $q \in S^{TT}_2(g)$.
\end{lem}

\begin{proof} Let $f:\Sigma \times (-\varepsilon,\varepsilon) \to M$ be a 
variation with velocity field $\partial_t f = \phi$. Then 
$$
\big(DG(f)\phi\big)_{ij} = 
\frac{\partial }{\partial t} h\big(\partial_i f,\partial_j f\big)
= h\big(\partial_i f,D_j \phi) + h(D_i \phi,\partial_j f).
$$
Working in normal coordinates for $g$ at $p \in \Sigma$, $t = 0$, 
we compute
\begin{eqnarray*}
g(DG(f) \phi, q) & = & 
2 h\big(\partial_i f,D_j \phi\big) q_{ij}\\
& = & 2 \partial_j \big(g^{ik} h(\partial_i f, \phi) q_{kj}\big) 
- 2 h\big(\underbrace{D_i \partial_j f}_{= A_{ij}}, \phi\big) q_{ij} 
- 2 h\big(\partial_i f,\phi\big) \underbrace{\partial_j q_{ij}}_{ = 0 }.
\end{eqnarray*}
The first term is a divergence which integrates to zero. Using in 
the second term that $q$ is tracefree, the claim follows.
\end{proof}

By Teichm\"uller theory, we have an $L^2$-orthogonal decomposition
$$
C^\infty({\rm Sym}^2(\Sigma_r)) = {\rm ker\,}D\pi_{\Sigma_r}(g) \oplus  S^{TT}_2(\Sigma_r).
$$
Furthermore we note that 
$$
\|q_1\|^2_{L^2(\Sigma_r)} = \|q_2\|^2_{L^2(\Sigma_r)} = 8\pi^2 rs \quad \mbox{ and } 
\quad \langle q_1,q_2 \rangle_{L^2(\Sigma_r)} = 0. 
$$
Hence we can compute
\begin{eqnarray*}
D\Pi(f_r) \phi & = & D\pi(g_{euc}) DG(f_r) \phi \\
& = & \frac{1}{8\pi^2 rs} \sum_{\mu = 1}^2  
\big\langle DG(f_r) \phi ,q_\mu \big\rangle_{L^2(\Sigma_r)} \,D\pi(g_{euc}) q_\mu\\
& = & - \frac{1}{16\pi^2 r^2} \sum_{\mu = 1}^2 
\Big(\int_{\Sigma_r} \big\langle g(A^\circ,q_\mu), \phi \big\rangle \,d\mu_g\Big)\,e_\mu\\
& = &  - \frac{1}{16 \pi^2 r^3 s} \Big(\int_{\Sigma_r} \langle \phi,\vec{n} \rangle\,d\mu_g\Big)\, e_2.
\end{eqnarray*}
In the last step, we used $A^\circ = \frac{1}{2rs} q_2 \otimes \vec{n}$. 
We see that $D{\cal A}(f_r) = 0$ as expected. Moreover, we have 
$$
\vec{W}(f_r) = \frac{r^2-s^2}{2 r^3s^3} \vec{n} \perp_{L^2} {\ker\,}D{\cal B}(f_r).
$$
This means that the CMC Clifford tori $f_r$ are actually critical points 
of the Willmore functional under the weaker constraint where 
${\cal B}(f) = s/r$ is prescribed. This suggests to first study the stability 
of the $f_r$ under the nondegenerate constraint ${\cal B}$.

\subsection{Simplified Constraint}
To use Lemma \ref{lemstability} we need to calculate the second 
variation both of the Willmore functional and of the constraint ${\cal B}$. 
We begin with the first. 

\begin{lem} \label{lemsvwillmore}
Let $\phi$ be a normal vector field along $f_r$. Then we have 
\begin{equation*}
D^2\cal W(f_r) (\phi, \phi) = \scalpr{L_r \phi}{\phi}_{L^2}
\end{equation*}
with
$$
L_r \phi =  
\frac{1}{2} \left(\Delta + \frac{1}{2r^2 s^2}\right) \Delta \phi 
+ \frac{1}{r^2} \nabla^{2}_{11} \phi + \frac{1}{s^2} \nabla^{2}_{22} \phi 
+ \frac{r^4 + s^4}{2r^{4} s^4} \phi.
$$
\end{lem}

\begin{proof}
Let $f(\cdot,t)$ be a variation with $f(\cdot,0) = f_r$ and $\partial_t f(\cdot,0) = \phi$. Then 
\begin{align*}
D^2\cal W (f_r)(\phi, \phi)
& = \frac{1}{2} \left( \int_{\Sigma_r} \scalpr{\nabla _t \vec W (f) \big|_{t=0}}{\phi} \, d\mu_g
+ \int_ {\Sigma_r} \scalpr{\vec W(f_r)}{\phi}\, \partial_t d\mu _{g} \big|_{t=0} \right).
\end{align*}
We refer to \cite{LMS}, equation $(33)$, for the second variation formula, 
compare also to \cite{GLW} and \cite{Lor}. In the case of constant mean 
curvature surfaces in ${\mathbb S}^3$,  the formula simplifies to
\begin{align*} 
\nabla _{t} \vec {W} 
= {} & (\Delta + |A^{\circ}| ^{2} - |\vec{H}|^{2}) (\Delta + |A|^{2} + 2) \phi \\
& + 2 \scalpr{\vec{H}} {A _{ij}} \nabla _{ij} ^{2} \phi + 2 \scalpr {A_{ik}} {\phi} A_{ij} \scalpr {A_{kj}} {\vec{H}} + 2  |H|^{2} \phi. 
\end{align*}
Here we take into account the area term in our definition of the Willmore 
functional, which is not included in the definition of the functional 
in \cite{LMS}, resulting in a slight difference of the two formulae. 
Plugging in the data of the tori yields
\begin{align*}
& \nabla _{t} \vec {W} (f_r) \\
= {} & \left(\Delta + \frac{1}{2r^{2} s^2} - \frac{(r^2-s^2)^2}{r^2 s^2}\right) \left(\Delta + \frac{s^2}{r^2} + \frac{r^2}{s^2} + 2\right) \phi \\
& + 2\left(\frac{s^2-r^2}{r^2} \nabla^{2}_{11} \phi + \frac{r^2-s^2}{s^2} \nabla^{2}_{22} \phi\right) \\
& + 2\left(\frac{s^2}{r^2} \frac{s^2-r^2}{r^2} + \frac{r^2}{s^2} \frac{r^2 - s^2}{s^2}\right) \phi 
+ 2 \Big(\frac{r^2}{s^2} + \frac{s^2}{r^2} - 2\Big)\,\phi \displaybreak[0] \\
= {} & \Delta^2 \phi + \Big(\frac{1}{2r^2 s^2} + 4 \Big)\, \Delta \phi\\
& + \Big(\frac{1}{2r^4} + \frac{1}{2s^4} + \frac{1}{r^2 s^2}\Big)\,\phi 
- \Big(\frac{r^2}{s^2} + \frac{s^2}{r^2} - 2\Big)\,\Big(\frac{r^2}{s^2} + \frac{s^2}{r^2} + 2\Big)\,\phi\\
& + \frac{2}{r^2} \nabla^2_{11} \phi + \frac{2}{s^2} \nabla^2_{22} \phi - 4 \Delta \phi\\
& + 2\Big(\frac{s^4}{r^4} + \frac{r^4}{s^4}\Big)\,\phi - 4\,\phi\\ 
= {} & \Delta^2 \phi + \frac{1}{2r^2s^2} \Delta \phi +  \frac{2}{r^2} \nabla^2_{11} \phi + \frac{2}{s^2} \nabla^2_{22} \phi\\
& + \frac{r^4 + s^4 + 2 r^2 s^2 - (2 r^8 + 2 s^8 - 4 r^4 s^4) + 4 s^8  + 4r^8  - 8 r^4 s^4}{2r^4 s^4}\,\phi\\
= {} &  \Delta^2 \phi + \frac{1}{2r^2s^2} \Delta \phi +  \frac{2}{r^2} \nabla^2_{11} \phi + \frac{2}{s^2} \nabla^2_{22} \phi
+ \frac{r^4 + s^4 + 2 r^2 s^2  - 4 r^4 s^4 + 2r^8 + 2s^8}{2r^4 s^4}\,\phi.  
\end{align*}
For the second term we get
$$
\scalpr {\vec{W}(f _{r})} {\phi}\, \partial _{t}d \mu _{g} 
= - \scalpr{\vec{W}(f_r)}{\phi}\,\scalpr{\vec{H}}{\phi}\,d\mu_g 
= - \big \langle \langle \vec{W}(f_r),\vec{H} \rangle \phi,\phi \big \rangle\,d\mu_g,
$$ 
where 
$$
\langle \vec{W}(f_r),\vec{H} \rangle  = \frac{(r^2-s^2)^2}{2r^4 s^4}. 
$$
Using that $r^2 + s^2 = 1$, we compute 
$$
r^8 + s^8 = \big((r^2+s^2)^2 - 2 r^2 s^2 \big)^2 - 2 r^4 s^4 = (1-2r^2 s^2)^2 - 2 r^4 s^4 = 1 - 4 r^2 s^2 + 2 r^4 s^4.
$$
The claim of the lemma now follows easily by adding the terms. 
\end{proof}

Now we turn to computing the second derivative of $\Pi = \pi \circ G$ in 
the direction of a normal vectorfield $\phi$ along $f_r$. We have  
\begin{equation} \label{eqsvprojection}
D^2\Pi(f_r) (\phi,\phi) = 
D^2\pi(g_{euc})\big(DG(f_r)\phi,DG(f_r)\phi\big) + D\pi(g_{euc})\,D^2G(f_r)(\phi,\phi). 
\end{equation}

We first work out the second term. 

\begin{lem} \label{lemsvmetric}
\begin{align*}
\langle D\pi(g_{euc}) D^2 G(f_r)(\phi,\phi), e_1 \rangle & = 
- \frac{1}{2\pi^2 r^2} \int_{\Sigma_r} \langle \nabla^2_{12} \phi,\phi \rangle\,d\mu_{g_{euc}},\\
\langle D\pi(g_{euc}) D^2 G(f_r)(\phi,\phi), e_2 \rangle & = 
\frac{1}{4\pi^2 r^2} \int_{\Sigma_r} \langle \nabla^2_{11} \phi -\nabla^2_{22} \phi,\phi \rangle\,d\mu_{g_{euc}} \\
& \quad + \frac{r^2-s^2}{4\pi^2 r^4 s^2} \int_{\Sigma_r} |\phi|^2 \,d\mu_{g_{euc}}.
\end{align*}
\end{lem}

\begin{proof}
Consider a variation $f(\cdot,t):\Sigma_r \to {\mathbb S}^3$, such that $\phi = \partial_t f$ 
is always normal along $f$. We have 
$$
D^2 G(f)(\phi,\phi) = \frac{\partial^2}{\partial t^2} G(f) - DG(f) D_t \phi. 
$$ 
By the first variation formulae for $g$ and $A$ we get 
\begin{align*}
\frac{\partial}{\partial t} G(f)_{ij} & =  -2 \langle A_{ij}, \phi \rangle,\\
\frac{\partial^2}{\partial t^2} G(f)_{ij} & =  
-2 \langle D_t A_{ij},\phi \rangle - 2 \langle A_{ij},D_t \phi \rangle\\
& =  - 2 \langle \nabla^2_{ij} \phi,\phi \rangle + 2 g^{kl} \langle A_{ik},\phi \rangle \langle A_{jl},\phi \rangle \\
& \quad - 2 \langle R^{{\mathbb S}^3}(\phi,\partial_i f)\partial_j f,\phi \rangle - 2 \langle A_{ij},D_t \phi \rangle.
\end{align*}
At $t=0$ we decompose $D_t \phi = (D_t\phi)^\perp + Df \cdot \xi$ at $t=0$ and let 
$\varphi_s:\Sigma \to \Sigma$ be the flow of the vector field $\xi$. Then 
\begin{eqnarray*}
DG(f) D_t \phi & = & DG(f) (D_t \phi)^\perp + DG(f) (Df \cdot \xi)\\
& = & - 2 \langle A, D_t \phi \rangle + \frac{\partial}{\partial s} G(f \circ \varphi_s)\Big|_{s=0}\\
& = &  - 2 \langle A, D_t \phi \rangle + \frac{\partial}{\partial s} \varphi_s^\ast f^\ast g_{{\mathbb S}^3}\Big|_{s=0}\\
& = &  - 2 \langle A, D_t \phi \rangle + {\cal L}_\xi (f^\ast g_{{\mathbb S}^3}).
\end{eqnarray*}
Inserting the geometric data of ${\mathbb S}^3$ and $f_r$ yields
$$
D^2 G(f_r)(\phi,\phi) = - 2 \langle \nabla^2 \phi,\phi \rangle + \frac{r^2-s^2}{r^2s^2} |\phi|^2 q_2 
+ \frac{(r^2-s^2)^2}{r^2s^2} |\phi|^2 g_{euc}  - {\cal L}_\xi g_{euc}.
$$
The two last terms vanish under $D\pi(g_{euc})$. Using that
$\|q_\mu\|_{L^2}^2 = 8\pi^2 rs$ and $D\pi(g_{euc}) q_\mu = \frac{s}{r} e_\mu$,
the claim of the lemma follows.  
\end{proof}

\begin{lem} \label{lemderivativedivergence}
On a two-dimensional manifold $\Sigma$, consider the operators
$$
\Lambda:{\cal M}(\Sigma) \times S_2(\Sigma) \to C^\infty(\Sigma),\,
\Lambda(g,q) = \tr_g q = g^{ij} q_{ij},
$$
$$
\theta:{\cal M}(\Sigma) \times S_2(\Sigma) \to \Omega^1(\Sigma),\,
\theta(g,q) = \div_g q = g^{ij} \nabla_i q(\partial_j, \cdot).
$$
Then for $h \in S_2(\Sigma)$ we have
\begin{eqnarray*}
D_1\Lambda(g,q) h & = & - g^{ij} g^{kl} q_{ik} h_{jl},\\
\big(D_1\theta(g,q) h\big)_m  & = & - g^{ij} g^{kl} h_{ik} (\nabla_j q)_{lm}
- g^{kl} (div_g h)_l q_{km} + \frac{1}{2} (\grad_g \tr_g h)^k q_{km}\\
&& - \frac{1}{2} g^{ij} g^{kl} (\nabla_m h)_{ik} q_{jl}.
\end{eqnarray*}
\end{lem}

\begin{proof} Let $g_t = g + t h$. We verify the equations using normal
coordinates with respect to $g$ at some point $p \in \Sigma$. First we have
\begin{align*}
D_1\Lambda(g,q) h = \frac{d}{dt} g^{ij} q _{ij}\Big|_{t=0} = - h_{ij} q_{ij}.
\end{align*}
In local coordinates the divergence is given by
$$
(\div_g q)_m
= g^{ij}(\nabla _i q)_{jm}
= g^{ij} (\partial_i q_{jm} - \Gamma_{ij}^k q_{km} - \Gamma_{im}^k q_{jk}).
$$
From the standard formula for the Christoffel symbols we get
$$
\partial _t \Gamma _{ij}^k \big|_{t=0} = \frac{1}{2} (\partial _i h _{jk} + \partial_j h_{ik} - \partial_k h_{ij}).
$$
We compute for the derivative of the divergence
\begin{eqnarray*}
\frac{d}{dt}(\div_g q)_m \Big|_{t=0} & = &
\frac{d}{dt}\left(g^{ij} (\partial_i q_{jm} - \Gamma_{ij}^k q_{km}
- \Gamma_{im}^k q_{jk})\right)\Big|_{t=0} \\
& = & - h_{ij} (\nabla_i q)_{jm}
- \frac{1}{2} (2\partial_i h_{ik} - \partial_k h_{ii})q_{km}\\
& & - \frac{1}{2} (\partial_i h_{mk} + \partial_m h_{ik} - \partial_k h_{im}) q_{ik}\\
& = &  - h_{ij} (\nabla_i q)_{jm} - (\nabla_i h)_{ik} q_{km} + \frac{1}{2} \partial_k (g^{ij}h_{ij}) q_{km}\\
&& - \frac{1}{2} (\nabla_m h)_{ik} q_{ik}.
\end{eqnarray*}
This proves the second formula.
\end{proof}

\begin{lem} \label{lemdependence}
For Riemannian metrics $g \in W^{k,2}(S_2(\Sigma))$ with $k \in {\mathbb N}$ 
sufficiently large, consider the operator 
$$
L_g: W^{k,2}(S_2(\Sigma)) \to W^{k,2}(\Sigma) \times W^{k-1,2}(T^\ast \Sigma),\,
L_g\, q = (\tr_g q, \div_g q).
$$
Fix $q_0 \in {\rm ker\,}L_{g_0}$. For $g$ close to $g_0$, there is a
unique $q \in {\rm ker\,}L_g$ such that $q-q_0 \perp {\rm ker\,}L_{g_0}$.
The function $q(g)$ is smooth and $\eta = Dq(g_0) \alpha \perp {\rm ker\,}L_{g_0}$ 
is characterized by the equations
\begin{eqnarray*}
\tr_g \eta & = & g^{ij} g^{kl} q_{ik} \alpha_{jl}\\
%\div_g \eta & = & g^{kl} (\div_g \alpha)_k q_{lm} - \frac{1}{2} (\grad_g \tr_g \alpha)^k q_{km}\\
%&& + g^{ij} g^{kl} \alpha_{ik} (\nabla_j q)_{lm} + \frac{1}{2} g^{ij} g^{kl} (\nabla_m \alpha)_{ik} q_{jl}.
%\end{eqnarray*}
%For $\eta^\circ = \eta - \frac{1}{2} (\tr_g \eta) g$ we conclude that 
%\begin{eqnarray*}
\big(\div_g \eta^\circ\big)_m & = & g^{kl} (\div_g \alpha)_k q_{lm} - \frac{1}{2} (\grad_g \tr_g \alpha)^k q_{km}\\
&& + g^{ij} g^{kl} \alpha_{ik} (\nabla_j q)_{lm} - \frac{1}{2} g^{ij} g^{kl} \alpha_{ik} (\nabla_m q)_{jl}.
\end{eqnarray*}
\end{lem}

\begin{proof} To each metric $g$ and each $q$ in $W^{k,2}(S_2(\Sigma))$, one associates 
the form
$$
\eta(X,Y) = q(X,Y) - iq(X,J_g Y), 
$$
where $J_g$ is the almost complex structure.
One checks that $\tr_g q = 0$ is equivalent to 
$q$ being complex bilinear with respect to $J_g$, and that further 
$\div_g q = 0$ reduces to the Cauchy-Riemann equation for $\eta$.
Since $(\Sigma,g)$ is biholomorphic to a standard torus, one concludes 
that $S^{TT}_2(g) = {\rm ker\,}L_g$ is two-dimensional and that
$L_g$ has closed range.\\
\\
Let $(\lambda,\xi) \in W^{k,2}(\Sigma) \times W^{k-1,2}(T^\ast \Sigma)$ 
be $L^2$-orthogonal to ${\rm im\,}L_g$, i.e.
$$
0 = \int_\Sigma \Big( \langle \lambda g, q \rangle_g + \langle \xi, \div_g q \rangle_g\Big)\,d\mu_g
\quad \mbox{ for all } q \in W^{k,2}(S_2(\Sigma)).
$$
This means $L_X g = \lambda g$ weakly, in other words 
$\lambda = \frac{1}{2} \div_g X$ and $X$ is a conformal 
Killing field. As this is again a Cauchy-Riemann equation, 
we get that $\big({\rm im\,}L_g\big)^{\perp_{L^2}}$ is also two-dimensional.\\
\\
Now let $g_0 \in W^{k,2}(S_2(\Sigma))$ be a fixed metric, with $L^2$ decompositions
$$
W^{k,2}(S_2(\Sigma)) = X_0 \oplus {\rm ker\,}L_{g_0} \quad \mbox{ and } \quad
W^{k,2}(\Sigma) \times W^{k-1,2}(T^\ast \Sigma) = {\rm im\,}L_{g_0} \oplus Y_0.
$$
With respect to these splittings, the operator $L_g$ is given by a matrix
$$
L_g = \left(\begin{array}{cc} A_g & B_g \\ C_g & D_g \end{array} \right).
$$
Clearly $A_{g_0}$ is an isomorphism while $B_{g_0},\,C_{g_0},\,D_{g_0}$ 
are zero. Now for $\phi = \varphi \oplus q_0$ the equation $L_g \phi = 0$ 
becomes 
$$
A_g \varphi + B_g q_0 = 0 \quad \mbox{ and } \quad C_g \varphi + D_g q_0 = 0.
$$
For $g$ sufficiently close to $g_0$, the operator $A_g$ is invertible.
The equations are then equivalent to 
$$ 
\varphi = - A_g^{-1} B_g q_0 \quad \mbox{ and } \quad 
(D_g - C_g A_g^{-1} B_g) q_0 = 0. 
$$
As the space of solutions is two-dimensional, the second equation must 
hold automatically, and the set of solutions is given by 
$ - A_g^{-1} B_g q_0 \oplus q_0$ where $q_0 \in S^{TT}_2(g_0)$. 
The formula for the derivative follows from the chain rule.
\end{proof}

Let $\alpha,\beta \in S_2(\Sigma_r)$ be symmetric forms satisfying 
$$
\tr_{g_{euc}} \alpha = \tr_{g_{euc}} \beta = 0 \quad \mbox{ and } \quad 
\beta \perp_{g_{euc}} S^{TT}_2(g_{euc}).
$$
Put $g(t) = g_{euc} + t\alpha$ and let $q^\mu(t) = q^\mu(g(t)) \in S^{TT}_2(g(t))$ 
be as in Lemma \ref{lemdependence}, that is 
$q^\mu(t) - q^\mu \perp_{g_{euc}} S^{TT}_2(g_{euc})$ and in particular $q^\mu(0) = q^\mu$. 
We expand 
$$
\beta = \beta_\mu(t) q^\mu(t) + \beta^\perp(t) \quad \mbox{ where }
\beta^\perp(t) \perp_{g(t)} S^{TT}_2(g(t)).
$$
By assumption $\beta_\mu(0) = 0$, and we have 
$$
D^2\pi(g_{euc})(\alpha,\beta) = \frac{d}{dt} D\pi(g(t)) \cdot \beta\Big|_{t=0}
= \beta_\mu'(0)\, D\pi(g_{euc}) \cdot q^\mu.
$$
Next, we compute in normal coordinates at $t= 0$ for symmetric $\alpha,\beta$
\begin{align*}
& \frac{d}{dt} g^{ij}(t) g^{kl}(t) \beta^\perp_{ik}(t)q^\mu_{jl}(t) \sqrt{\det g(t)}\Big|_{t=0}\\
& = (\beta^\perp)'_{ik} q^\mu_{ik} + \beta^\perp_{ik} (q^\mu)'_{ik} 
- 2 \alpha_{ij} \beta^\perp_{ik} q^\mu_{jk} 
- \frac{1}{2} \beta^\perp_{ik} q^\mu_{ik} \tr_{g_{euc}} \alpha. 
\end{align*}
Using that $\alpha,\beta$ are tracefree, we see that 
$$
\alpha_{i1} \beta_{i1} = \alpha_{i2} \beta_{i2} \quad \mbox{ and } \quad
\alpha_{i1} \beta_{i2} = - \alpha_{i2} \beta_{i1}.
$$
As $q^\mu$ is also symmetric and tracefree, we obtain 
$\alpha_{ij} \beta^\perp_{ik} q^\mu_{jk} = 0$. Integrating and using
$\langle q^\lambda,q^\mu \rangle_{L^2(g_{euc})} = 8 \pi^2 rs\, \delta_{\lambda \mu}$
we conclude
\begin{eqnarray*} 
8\pi^2 rs\, \beta_\mu'(0) & = &
- \langle \big(\beta - \beta_\lambda q^\lambda \big)'(0), q^\mu \rangle_{L^2(g_{euc})}\\
& = & - \langle (\beta^\perp)'(0),q^\mu \rangle_{L^2(g_{euc})}\\ 
& = & \langle \beta, (q^\mu)'(0) \rangle_{L^2(g_{euc})}.
\end{eqnarray*}
Let $\eta = (q^\mu)'(0) \perp_{L^2} S^{TT}_2(g_{euc})$. Then by 
Lemma \ref{lemdependence} we get using $\tr_{g_{euc}} \alpha = 0$
$$
(\div_{g_{euc}} \eta^\circ)_l = (\div_{g_{euc}} \alpha)_k q^\mu_{kl}.
$$
Let us focus on the case $\mu =2$, which will be the relevant one. Then 
\begin{eqnarray*} 
%(\div_{g_0} \eta^\circ)_1 = (\div_{g_{euc}} \alpha)_2, \quad  
%(\div_{g_{euc}} \eta^\circ)_2 = (\div_{g_{euc}} \alpha)_1 & \mbox{ for } & \mu = 1,\\
(\div_{g_{euc}} \eta^\circ)_1 & = & - (\div_{g_{euc}} \alpha)_1,\\
(\div_{g_{euc}} \eta^\circ)_2 & = & (\div_{g_{euc}} \alpha)_2 
\end{eqnarray*}
Now for $\alpha = \alpha_1 q^1 + \alpha_2 q^2$ we have
$$
(\div_{g_{euc}} \alpha)_1  = \partial_2 \alpha_1 - \partial_1 \alpha_2, \quad
(\div_{g_{euc}} \alpha)_2  = \partial_1 \alpha_1 + \partial_2 \alpha_2.
$$
Putting $\eta^\circ = u_1 q^1 + u_2 q^2$, the equations become
%$$
%\partial_2 u_1 - \partial_1 u_2 = \partial_1 \alpha_1 + \partial_2 \alpha_2 \quad 
%\partial_1 u_1 + \partial_2 u_2 = \partial_2 \alpha_1 - \partial_1 \alpha_2 
%\quad \mbox{ for } \mu = 1,
%$$
\begin{eqnarray*}
\partial_2 u_1 - \partial_1 u_2 & = &  \partial_1 \alpha_2 - \partial_2 \alpha_1\\ 
\partial_1 u_1 + \partial_2 u_2 & = & \partial_1 \alpha_1 + \partial_2 \alpha_2.
\end{eqnarray*}
Differentiating and combining the equations yields 
\begin{eqnarray*}
\Delta u_1 & = &  (\partial_1^2 - \partial_2^2) \alpha_1 + 2 \partial^2_{12} \alpha_2\\
\Delta u_2 & = &  2 \partial^2_{12} \alpha_1 - (\partial_1^2 - \partial_2^2) \alpha_2.
\end{eqnarray*}
We specialize to $\alpha = \varphi q^2$, $\beta = \psi q^2$. 
Then $D^2\pi(g_{euc})(\alpha,\beta) = \beta_2'(0) \frac{s}{r} e_2$, where 
$8\pi^2 rs \beta_2'(0) = \langle \psi q^2, \eta^\circ \rangle_{L^2(g_{euc})}$. 
Here $\eta^\circ = u_1 q^1 + u_2 q^2$, and $u_2$ is determined by the equation
$$
\Delta u_2 = - (\partial_1^2 - \partial_2^2) \varphi.
$$
To determine $u_2$ we specialize further to functions $\varphi$ in a Fourier space.

\begin{dfi}
Put $V_k = \{\cos (kx), \sin (kx)\}$ for $k \in \N$, and $V_0 = \{1\}$.
We denote by $\cal A_{k,l}(\Sigma_r)$ the set of functions
$w(x,y) = u(\frac{x}{r}) v(\frac{y}{s})$ on $\Sigma_r$, where
$u \in V_k,\,v \in V_l$. The set
$$
\cal F (\Sigma_r) := \bigcup _{k,l\in \N_0} \cal A_{k,l}(\Sigma_r)
$$
is the Fourier basis on the torus $\Sigma_r$.
\end{dfi}

For $\varphi \in {\cal A}_{kl}$ with $(k,l) \neq (0,0)$ we obtain the solution 
\begin{equation}
\label{eqconstants}
u_2 = - c_r(k,l) \varphi \quad \mbox{ where } c_r(k,l) = \frac{k^2 s^2 - l^2 r^2}{k^2 s^2 + l^2 r^2}.
\end{equation}
Note that $\langle \eta^\circ, q^2  \rangle_{L^2(g_{euc})} = 0$ with this choice of 
the integration constant. Now
$$
8\pi^2 rs \beta_2'(0) = - 2 c_r(k,l) \, \int_{\Sigma_r} \varphi \psi\,d\mu_{g_{euc}}.
$$ 
Summarizing we have for $\alpha = \varphi q^2$, $\varphi \in {\cal A}_{kl}$, and $\beta = \psi q^2$, 
$$
D^2 \pi ^2(g_{euc}) (\alpha, \beta) = - \frac{c_r(k,l)}{4\pi^2 r^2}\, 
\Big(\int_{\Sigma_r} \varphi \psi\,d\mu_{g_{euc}}\Big)\,e_2.
$$

%$DG_{f_r} (\phi) = -2 \scalpr{A}{\Phi}$ into its tracefree and pure trace part.
%Note that the tracefree part is of the form $\lambda q^2$ for some
%function $\lambda:\Sigma \to \R$, because $ A_{12} = A _{21} = 0$.

\begin{lem} \label{lemsvprojection}
For normal vector fields $\Phi = \varphi \vec n$, $\Psi = \psi \vec n$ with 
$\varphi \in {\cal A}_{k,l}(\Sigma_r)$ and $\psi \in {\cal A}_{m,n}(\Sigma_r)$,
we have the formula
$$
D^2 \pi ^2 ({g_{euc}}) (DG(f_r)\Phi, DG(f_r)\Psi) =
- \frac{2(r^2-s^2) + c_r(k,l)}{4\pi^2 r^4 s^2} \int_{\Sigma_r} \varphi \psi \,d\mu_{g_{euc}},
$$
where the constants $c_r(k,l)$ are as in (\ref{eqconstants}). 
%\begin{align*}
%& D^2 \pi ^2 ({g_{euc}}) (DG_{f_r}(\Phi_{k,l}), DG_{f_r}(\Phi_{k,l})) \\
%= {} & - \frac{2(2 r^2 - 1)+ c(k,l;r)}{4 \pi ^2 r^4s^2} 
%\int _{\Sigma_r} \scalpr{\Phi_{k,l}}{\Phi_{k,l}} \, d\mu _g 	
%\end{align*}
%with
%$$
%c(k,l;r)= \frac{k^2s^2-l^2 r^2}{k^2 s^2 + l^2 r^2} .
%$$
\end{lem}

\begin{proof}
Let $\alpha = \scalpr{A}{\Phi}$, $\beta = \scalpr{A}{\Psi}$, and decompose
$\alpha, \beta$ into its trace and trace free components, respectively: 
$\alpha = \alpha^{c} + \alpha^\circ$, $\beta = \beta^{c} + \beta^\circ$. 
Recall from \cite[Proposition A.7]{KS} that 
$$
D^2 \pi (g) (h, \sigma g + \mathcal{L}_X g) = - D\pi (g) (\sigma h + \mathcal{L}_X h)
$$
for $h \in S_2(\Sigma)$, $\sigma \in C^{\infty}(\Sigma_r)$ and $X \in C^\infty(T\Sigma)$.
Writing $\alpha^{c} = \lambda_\alpha g_{euc}$, $\beta^{c} = \lambda_\beta g_{euc}$ we 
thus get 
$$
D^2 \pi^2 (g_{euc}) (\alpha^c, \beta^c) = - D \pi^2 (g_{euc}) (\lambda_\alpha \lambda_\beta g_{euc}) = 0.
$$
As $A_{12} \equiv 0$ for $f_r$, we have $\alpha^\circ = f_\alpha q^2$ and $\beta^\circ = f_\beta q^2$. 
This yields 
\begin{eqnarray*}
D^2 \pi^2 (g_{euc})(\alpha^c,\beta^\circ)&  = & - D\pi^2 (g_{euc}) (\lambda_\alpha \beta^\circ)\\ 
& = &  - \langle \lambda_\alpha \beta^\circ, q^2 \rangle_{L^2(g_{euc})} \frac{ D\pi^2(g_{euc}) q^2}{\|q^2\|_{L^2(g_{euc})}^2}\\
& = & - \frac{1}{4\pi^2 r^2} \int_{\Sigma_r} \lambda_\alpha f_\beta \,d\mu_{g_{euc}}\\
& = & - \frac{r^2-s^2}{16 \pi^2 r^4 s^2} \int_{\Sigma_r} \varphi \psi\,d\mu_{g_{euc}}.
\end{eqnarray*}
Here we used that $A$ is given explicitely, which yields 
$$
\lambda_\alpha = \frac{r^2 - s^2}{2rs} \varphi, \quad f_\beta = \frac{1}{2rs} \psi. 
$$
%\begin{align*}
%& D^2 \pi^2 (g_{euc})(\alpha,\beta) \\
% = {} & D^2 \pi ^2 (g_{euc})(\alpha ^c,\beta ^c) +  D^2 {\pi} ^2 (g_{euc})(\alpha^c,\beta^s) +
%D^2{\pi} ^2 (g_{euc})(\alpha^s,\beta^c) 
%+ D^2 {\pi} ^2 (g_{euc})(\alpha^s,\beta^s).
%\end{align*}
Finally, our above computation yields for the tracefree components that 
$$
D ^2{\pi}^2 (g_{euc}) (\alpha^\circ, \beta^\circ) =  
- \frac{c_r(k,l)}{16 \pi^2 r^4 s^2} \int_{\Sigma_r} \varphi \psi \,d\mu_{g_{euc}}. 
$$
\end{proof}  
We can finally calculate $D^2 \cal B (f_r)(\Phi, \Phi)$. In the following, we denote by 
$\frak X (f_r)$ the vectorfields along $f_r$, i.e. the sections of $f_r ^* T\Ss^3$, and by
$\frak X(f_r)^\perp$ the normal vectorfields along $f_r$. 
\begin{lem} \label{lemsvprojection2}
Let   $L_r ^{\cal B}: \frak X (f_r) ^\perp \rightarrow \frak X (f_r) ^\perp$ be given by
\begin{equation*}
L_r ^{\cal B}(\Phi _{k,l}) = - \frac {1} {4 \pi ^2 r^2} \left(\nabla^2 _{22}  - \nabla^2 _{11}  + \frac{r^2-s^2 + c(k,l;r) }{r^2s^2}  \right)\Phi_{k,l}
\end{equation*}
for 
$$
\Phi _{k,l} = \varphi _{k,l} \vec n, \quad \varphi _{k,l} \in \cal A _{k,l} (\Sigma_r)
$$
with
$$
c_r(k,l) = \frac{k^2s^2-l^2 r^2}{k^2s^2+l^2 r^2}.
$$
Then for a normal vectorfield $\Phi \in \frak X (f_r) ^\perp$ along $f_r$ we have
\begin{equation*}
D^2 \cal B (f_r)(\Phi, \Phi) = \scalpr{L^{\cal B} _r (\Phi)}{\Phi}_{L^2}.
\end{equation*}
\end{lem}
\begin{proof}
By equation \eqref{eqsvprojection} we have
\begin{equation*}
D^2 \cal B (f_r) (\Phi,\Phi) = D^2 \pi^2 (g_{euc}) (DG (f_r) \Phi, DG (f_r) \Phi) 
+ D \pi ^2 (g_{euc}) D^2 G (f_r)(\Phi,\Phi).
\end{equation*}
These two terms were calculated in Lemma \ref{lemsvmetric} and 
Lemma \ref{lemsvprojection}. 
Plugging in yields for $\Phi _{k,l} = \varphi _{k,l} \vec n$, $\varphi _{k,l} \in \cal A _{k,l}$: 
\begin{align*}
& D^2 \cal B  _{f_r} (\Phi_{k,l},\Phi_{k,l})	\\
= {} & -  \frac{2(r^2-s^2)+ c(k,l;r)}{4 \pi^2 r^4 s^2} 
\int _{\Sigma_r} \scalpr{\Phi_{k,l}}{\Phi_{k,l}} \, d\mu _g \\
& + \frac{1}{4\pi^2 r^2} \int \left \langle \left(\nabla^2 _{11} - \nabla^2 _{22} + \frac{r^2-s^2}{r^2s^2}\right) \Phi_{k,l}, \Phi_{k,l} \right \rangle \, d\mu _g \\
= {} & - \frac{1}{4\pi^2 r^2} \int _{\Sigma _r} \left \langle \left(\nabla^2 _{22}  - \nabla^2 _{11}  +\frac{r^2-s^2+ c(k,l;r)}{r^2s^2} \right) \Phi_{k,l}, \Phi_{k,l} \right \rangle \, d\mu _g.
\end{align*}
For $\Psi _{u,v} =\psi _{u,v} \vec n$, $\psi _{u,v} \in \cal A _{u,v} (\Sigma_r)$, we get
$$
D^2 \cal B (f_r) (\Phi_{k,l},\Psi_{u,v}) = 0
$$
if $\Psi _{u,v} \neq \Phi _{k,l}.$
\end{proof}
We can now prove our first stabiltiy theorem:
\begin{thm} \label{thmstability1} 
The tori $f _r : \Sigma _r \rightarrow \Ss^3$ are Willmore stable wrt. the constraint 
$\cal B$ if and only if 
$$
r \in \left[\frac{1}{2}, \frac{\sqrt{3}}{2} \right] 
\quad \mbox{or, equivalently,} \quad
\cal B (f_r) \in \left[\frac{1}{\sqrt{3}}, \sqrt{3} \right].
$$ 
More precisely, let
$L_r: (\frak X (f_r))^{\perp} \rightarrow (\frak X(f_r)) ^{\perp}$ be given by
\begin{align*}
L_r(\Phi _{k,l}) = {} & \frac{1}{2} \biggl(\left(\Delta + \frac{1}{r^2s^2}\right) \Delta  + \frac{1}{r^2}\nabla^2 _{11} + \frac{1}{s^2}\nabla^2 _{22}  + \frac{1 - (r^2-s^2) c(k,l;r)}{2r^4s^4}  \biggr) \Phi _{k,l} 
\end{align*}
for
$$\Phi _{k,l} = \varphi _{k,l} \vec n,  \quad \varphi_{k,l} \in \cal A _{k,l}(\Sigma_r)$$
with
$$
c_r(k,l) = \frac{k^2s^2-l^2 r^2}{k^2s^2+l^2 r^2}.
$$
Then for an admissible variation $h: \Sigma _r \times (-\delta, \delta) \rightarrow \Ss^3$ 
of $f_r$  with $\partial _t h \big|_{t=0} = \Phi$ we have
$$
\frac{d^2}{dt^2} \cal W (h) \Big|_{t=0} = \scalpr{L _r (\Phi^\perp)}{\Phi^\perp} _{L^2}.
$$
Let $k \geq 1$ and $2\leq l \leq k+1$. Then $L_r$ has exactly $k$ negative eigenvalues for 
\begin{gather*}
r \in \left[\frac{1}{k+2}, \frac{1}{k+1}\right) 
\quad \mbox{resp.} \quad 
\cal B (f_r) \in \left(\left((k+1)^2 -1\right)^{\frac{1}{2}}, \left((k+2)^2 - 1\right)^{\frac{1}{2}} \right].
\end{gather*}
The eigenspaces are $2$-dimensional and the $2k$ corresponding eigenfunctions are given by
$$
\Phi _l (x,y) = \sin \left(\frac{l}{s} y \right) \vec n (x,y) \in \frak X (f_r)^\perp 
$$
and
$$
\Psi_l(x,y) =\cos \left(\frac{l}{s}y\right) \vec n(x,y) \in \frak X(f_r)^\perp.
$$ 
\end{thm}

\begin{proof}
Let $h: \Sigma \times (-\delta, \delta) \rightarrow \Ss^3$ be an admissible variation of $f_r$ with
$\partial _t h \big|_{t=0} = \Phi$. 
By reparametrizing by an innner differomorphism $\varphi: \Sigma _r \times (-\delta, \delta) \rightarrow \Sigma_r$
we can achieve that $\tilde h (p,t) : = h(\varphi(p,t),t)$ is always normal with
 $\partial_t \tilde h \big|_{t=0} = \Phi^\perp$. Note that $\tilde h$ is admissable by definition of the Teichm\"uller 
 space.
By Lemma \ref{lemstability} we have
\begin{align*}
\frac{d^2}{dt^2} \mathcal{W}(h_t) \Big|_{t=0}  & = \frac{d^2}{dt^2} \mathcal{W}(\tilde h_t) \Big|_{t=0} \\
& = D^2\cal W (f_r) (\Phi^\perp, \Phi^\perp) - \lambda D^2 \cal B (f_r)(\Phi^\perp, \Phi^\perp)
= \scalpr{L _r (\Phi^\perp)}{\Phi^\perp} _{L_2}
\end{align*}
with
$$
L_r = L^{\cal W}(f_r) - \lambda L^{\cal B}(f_r).
$$
The Lagrange mulitplier $\lambda \in \R$ is given by $\lambda = D\cal W  (f_r) (\Theta)$ for any vector field 
$\Theta\in \frak X(f_r)$ along $f_r$ with $D\cal B (f_r) (\Theta) = 1$. 
For $r= \frac{1}{\sqrt{2}}$  we have $D\cal W  (f_r) \equiv 0$ and hence $\lambda =0$. 
For $r \neq \frac{1}{\sqrt{2}}$ we have $0\neq \vec W (f_r) \perp \ker D\cal B (f_r)$ and we put
$$
\Theta = \frac{\vec W (f_r)} {D\cal B_{f_r} (\vec W (f_r))}.
$$
We get
\begin{equation*}
\lambda = D\cal W (f_r) (\Theta) = \frac{1}{2}\scalpr{\vec W  (f_r)}{\Theta} _{L^2}
= \frac{\|\vec W (f_r)\|^2}{2 D\cal B (f_r) (\vec W (f_r))}.
\end{equation*}
One calculates
$$
\|\vec{W}(f_r)\|^2 = \int _{\Sigma _r} \frac{(r^2-s^2)^2}{4r^6s^6} \,d\mu_g 
= \frac{(r^2-s^2)^2}{4r^6s^6}|\Sigma _r| 
$$
and
$$
D\cal B  _{f_r}(\vec{W}(f_r))=\frac{1}{8\pi^2 r^2} \scalpr{-2\scalpr{A}{\vec{W}(f_r)}}{q^2} 
=-  \frac{r^2-s^2}{ 8 \pi ^2 r^6 s^4} |\Sigma _r|.
$$
Thus
$$
\lambda = - \frac{\pi^2(r^2-s^2)}{s^2}.
$$

Next, we show that for
$\Phi _{k,l} = \varphi _{k,l} \, \vec n,$ $\varphi _{k,l} \in \cal A_{k,l}(\Sigma_r)$,
\begin{equation} \label{eqlrphikl}
\begin{split}
L_r(\Phi _{k,l}) = {} & \frac{1}{2} \biggl(\left(\Delta + \frac{1}{r^2s^2}\right) \Delta  + \frac{1}{r^2}\nabla^2 _{11} \\
& \qquad + \frac{1}{1-r^2}\nabla^2 _{22}  + \frac{1 - (r^2-s^2) c(k,l;r)}{2r^4s^4}  \biggr) \Phi _{k,l}.
\end{split}
\end{equation}
This follows by plugging in Lemma \ref{lemsvwillmore} and Lemma \ref{lemsvprojection2}:
\begin{align*}
& L_r (\Phi_{k,l}) \\
= {} & \left(L _r ^{\cal W} - \lambda L _r ^{\cal B} \right)(\Phi_{k,l}) \\
= {} & \frac{1}{2} \left[ \left( \left(\Delta + \frac{1}{2r^{2} s^2 }\right) \Delta   + \frac{2}{r^{2}} \nabla ^{2} _{11}  + \frac {2} {s^2} \nabla ^{2} _{22}  + \frac{1 - 2r^{2} + 2 r^{4}} {r^{4} s^4} \right) \Phi _{k,l} \right] \\
&  + \frac{\pi^2(r^2-s^2)}{s^2} \left( - \frac{1}{4\pi^2 r^2} \left(\nabla^2 _{22}  - \nabla^2 _{11} 
 + \frac{r^2-s^2 + c(k,l;r) }{r^2s^2}  \right) \right)\Phi _{k,l} \\
= {} & \frac{1}{2} \biggl(\left(\Delta + \frac{1}{2r^{2} s^2 }\right) \Delta  + \frac{3 -2r^2}{2r^2s^2} \nabla^2 _{11}  \\
& \qquad + \frac{1 +2r^2}{2r^2s^2} \nabla^2 _{22} 
 + \frac{1 - (r^2-s^2)c(k,l;r)}{2r^4s^4} \biggr) \Phi _{k,l}\\
= {} & \frac{1}{2} \biggl( \left(\Delta + \frac{1}{r^2s^2}\right) \Delta + \frac{1}{r^2}\nabla^2 _{11} + \frac{1}{s^2}\nabla^2 _{22}  + \frac{1 - (r^2-s^2)c(k,l;r)}{2r^4s^4} \biggr) \Phi _{k,l}\\
\end{align*}

Now we turn to the calculation of the eigenvalues of $L_r$. 
Plugging in
\begin{gather*} 
\nabla^2 _{11} \Phi _{k,l} = -\frac{k^2}{r^2} \Phi_{k,l}, \quad \nabla^2_{22} \Phi_{k,l} = -\frac{l^2}{s^2} \Phi_{k,l}, 
\quad \Delta \Phi_{k,l} =-(\frac{k^2}{r^2} +\frac{l^2}{s^2})\Phi_{k,l}
\end{gather*}
and
$$ 
c_r(k,l) = \frac{ k^2s^2-l^2 r^2}{k^2s^2+l^2 r^2}
$$
into \eqref{eqlrphikl} yields
\begin{align*} 
L_r(\Phi_{k,l}) = {} & \biggl(\frac{(k^2s^2+l^2 r^2)^2}{r^4s^4} - \frac{k^2s^2+l^2 r^2}{r^4s^4} - \frac{k^2}{r^4} \\
& -\frac{l^2}{s^4}  + \frac{1}{2r^4s^4} 
- \frac{(r^2-s^2)(k^2s^2-l^2r^2)}{2r^4s^4(k^2s^2+l^2r^2)} \biggr) \Phi_{k,l} \\ \displaybreak[0] 
={} & \frac{N(k,l;r)}{r^4s^4(k^2s^2+l^2r^2)}
\end{align*}
with
\begin{align*}
N(k,l;r) = {} & \frac{1}{2} \bigl[ 2 (k^2s^2+l^2 r^2)^2-2(k^2s^2+l^2 r^2)-2k^2s^4 \\
& \quad -2l^2 r^4 + 1 \bigr]\left(k^2s^2+l^2r^2\right) -\frac{1}{2} (r^2-s^2)\left(k^2s^2-l^2r^2\right) \\
= {} & \bigl[(k^2s^2+l^2 r^2)^2-(k^2s^2+l^2 r^2) -k^2s^4 -l^2 r^4\bigr](k^2 s^2+l^2r^2) \\
& + k^2s^4 +  l^2 r^4 \\
= {} & \bigl[(k^4-k^2)s^4 + (l^4 - l^2) r^4 + 2k^2l^2 r^2s^2  - k^2s^2 - l^2 r^2\bigr] \left(k^2s^2+l^2r^2\right)\\
&  + k^2s^4 +  l^2 r^4 \\
= {} & \bigl[ (k^4-k^2)s^4 + (l^4 - l^2) r^4  + 2k^2l^2 r^2s^2\bigr] \left(k^2s^2+l^2r^2\right) \\
&+(k^2-k^4)s^4+  (l^2- l^4) r^4-2l^2k^2r^2s^2 \\
= {} & \bigl[ (k^4-k^2)s^4+ (l^4 - l^2) r^4  + 2k^2l^2 r^2s^2\bigr] \left( k^2s^2+l^2r^2 -1 \right).
\end{align*}
Thus
\begin{equation*}
L_r(\Phi _{k,l}) = E(k,l;r)\Phi _{k,l}
\end{equation*}
with
$$E(k,l;r) = \frac{N(k,l;r)}{r^4s^4(k^2s^2+l^2r^2)},
$$
\begin{align*}
N(k,l;r) = {} & \bigl[ (k^4-k^2)s^4+ (l^4 - l^2) r^4  + 2k^2l^2 r^2s^2\bigr] \left( k^2s^2+l^2r^2 -1 \right).
\end{align*}
We get
$E(1,0;r) = E(0,1;r) = 0.$ \\
Because of
$$
(k^4-k^2)s^4+ (l^4 - l^2) r^4 + 2k^2l^2 r^2s^2 \geq 0
$$
and
$$
r^4s^4(k^2s^2+l^2r^2) > 0
$$
for all $r \in (0,1)$, $(k,l) \in \mathbb{Z}^2\setminus \{0\}$ the sign of $E(k,l;r)$ is determined by
$$k^2s^2+l^2r^2 -1. $$ 

If $k \neq 0$, $l\neq 0$ 
$$k^2s^2+l^2r^2 -1 \geq (\min \{k,l\})^2 -1 \geq 0.$$ \\
Hence, let $l=0$. 
Then
\begin{equation} 
k^2s^2 - 1 \geq 0 \; \Leftrightarrow \; r \leq \sqrt{\frac{k^2-1}{k^2}}.
\end{equation}
For $k=0$ we have
\begin{equation} 
l^2r^2 - 1 \geq 0 \; \Leftrightarrow \; r \geq \frac{1}{l},
\end{equation}
thus for $r \in [\frac{1}{2}, \frac{1}{2}\sqrt{3}]$ all eigenvalues are non-negative and for
$$
r \in \left[ \frac{1}{k+1}, \frac{1}{k}\right)$$ 
we have exactly $k-1$ negative eigenvalues. 
\end{proof}

\subsection{Full Constraint}
As a direct corollary from Theorem \ref{thmstability1} we get
\begin{cor}
The tori  $T_r = r\,{\mathbb S}^1 \times s \, {\mathbb S}^1 \subset {\mathbb S}^3$, 
$r^2 + s^2 = 1$, are stable constrained Willmore surfaces for $r\in \left[\frac{1}{2}, \frac{\sqrt 3}{2}\right]$.
\end{cor}
\begin{proof} 
A  variation $h:\Sigma \times (-\delta, \delta) \rightarrow \Ss^3$ with $\Pi (h) \equiv \Pi(f_r)$
also satisfies $\cal B  (h) \equiv \cal B (f_r)$.
\end{proof} 

The main idea to adress the problem of the full constraint $\Pi$ is that it is enough to restrict to the class of surfaces of revolution. 
\begin{dfi} 
Let $\Sigma= \R^2/ (2\pi \, \Z \times 2\pi \, \Z)$ be the square torus. For $\gamma: \Ss^1 \rightarrow \left(0, \frac{\pi}{2} \right)$ we call the immersion
$$
h _\gamma: \Sigma \rightarrow \Ss^3, 
\quad h_\gamma (u,v) = \left(\cos \left(\gamma(v)\right) e^{iu}, \sin(\gamma(v)) e^{iv} \right)
$$
the corresponding surface of revolution. 
We denote by
\[
\mathcal C_{rot} = \{h_ \gamma \in {\rm Imm}^k(\Sigma, \Ss^3) \, | \; \gamma: \Ss^1 \rightarrow \left(0,\frac{\pi}{2} \right) \}.
\]
the class of surfaces of revolution.
\end{dfi}
\begin{thm} \label{thmstability2}
The tori $T_r = r\,{\mathbb S}^1 \times s \, {\mathbb S}^1 \subset {\mathbb S}^3$, 
$r^2 + s^2 = 1$, are stable constrained Willmore surfaces if and only if 
$$
r \in \Big[\frac{1}{2},\frac{\sqrt{3}}{2}\Big] \quad \mbox{ or equivalently } \quad
b \in \Big[\frac{1}{\sqrt{3}},\sqrt{3}\Big]. 
$$
Here $a+ib \in {\mathbb H}$ are standard coordinates on the Teichm\"uller 
space of the torus. 

More precisely, let $k \geq 1$ and $2\leq l \leq k+1$. For
\begin{gather*}
r \in \left[\frac{1}{k+2}, \frac{1}{k+1}\right)  \quad \mbox{resp.} \quad
\cal B (f_r) \in \left(\left((k+1)^2 -1\right)^{\frac{1}{2}}, \left((k+2)^2 - 1\right)^{\frac{1}{2}} \right]
\end{gather*}
we find for all instable directions $\Phi _l, \Psi_l \in \frak X (f_r)$ from Theorem \ref{thmstability1} admissible variations
$h: \Sigma _r \times (-\delta, \delta) \rightarrow \Ss^3$ of $f_r$ 
in the class $\cal C _{rot}$ of surfaces of revolution with  $\partial _t h \big|_{t=0} = \Phi_l$ resp. $\partial _t h \big|_{t=0} = \Psi _l$. We have
$$
\frac{d^2}{dt^2} \cal W _\varkappa (h) \Big|_{t=0} = \scalpr{L _r (\Phi_l)}{\Phi_l} _{L^2},
$$
where $L_r$ is defined as in Theorem $\ref{thmstability1}$. 
\end{thm}

\begin{proof} 
Obviously the tori $T_r$ can be parametrized by $h _\gamma \in \mathcal C$ with $\gamma \equiv \rho := \arccos r$.

Consider any $h_\gamma \in \mathcal C$. By reparametrization of $\gamma$ we get a conformal immersion on the rectangle spanned by $(2 \pi, 0)$ and $(0, 2 \pi \omega)$. 
Hence $\Pi (h _\gamma) = (0, \omega)$, in particular $\cal A (h _\gamma) = 0$ for all surfaces of revolution.  

Now let $r \in \left[\frac{1}{k+2}, \frac{1}{k+1}\right)$ be fixed, $\rho := \arccos(r)$ and $h_\rho:[0, 2 \pi] ^2 \rightarrow \Ss^3$ the above parametrization of $T_r$.
In the proof of Theorem \ref{thmstability1} we have found for $2\leq l \leq k+1$ the instable directions
$$
\Phi _l \in \frak X (h _\rho), \; \Phi (u,v) = \sin (lv) \vec n(u,v)
$$  
and
$$
\Psi _l \in \frak X(h_\rho), \; \Psi(u,v) = \cos(lv) \vec n(u,v).
$$
The normal of a surface of revolution is given by
$$
\vec n (u,v)= \left( \sin (\gamma(v)) e^{iu}, 
 - \cos(\gamma(v)) e ^{iv} - i \frac{\gamma'(v)}{sin(\gamma(v))} e^{ iv}\right).$$ 
 We vary $h _\rho$ in direction of $\Phi_l$ in $\mathcal C$ by $f _t = h _{\gamma(t)}$ with
\begin{gather*}
\gamma : \Ss^1 \times (-\delta, \delta) \rightarrow \left(0,\frac{\pi}{2}\right), \\
\gamma(u,t) = \rho - t \sin (lv).
\end{gather*}
We have
\begin{align*}
\partial _t f (u,v) \big|_{t=0} & = \left(\sin (\rho) \sin (lv) e^{iu}, -\cos (\rho) \sin (lv) e^{iv} \right) 
= \Phi _l
\end{align*}
as desired. This variation already satisfies  $\cal A (f(t)) \equiv 0 = \cal A (h _\rho)$.
$\cal B$ is non-degenerate, hence we can correct this coordinate by using the implicit function theorem. 
Consider the $2$-parameter-family of surfaces of revolution given by
\begin{gather*}
\tilde f: \Sigma \times (-\delta, \delta)^2 \rightarrow \Ss^3, \\
\tilde f (s,t) = h _{s + \gamma (t)}.
\end{gather*}
Here $s + \gamma (t): \Ss^1 \rightarrow (0,\frac{\pi}{2})$, 
$(s+ \gamma (t))(v) = s + \gamma (t) (v) $ is well-defined for $t$ and $s$ small enough. 
We calculate
$$
\partial _s \tilde f\big|_{s=t=0} (u,v) = \left( - \sin (\rho) e^{iu}, \cos (\rho) e^{iv} \right) = - \vec n \in \frak X(h_\rho).
$$
Consider
$$
G : (-\delta, \delta) ^2 \rightarrow \R, \; G(s,t) := \cal B (\tilde{f} (s,t)).
$$
Then
$$
\partial _s G \big|_{s=t=0} = D \cal B _{h_\rho} ( - \vec n) \neq 0.
$$
Hence, by the implicit function theorem, there exists a function (after eventually making $\delta$ smaller) 
$\tau: (-\delta, \delta) \rightarrow (-\delta, \delta)$ with
$$
G(\tau(t), t) \equiv G(0,0).
$$
Thus $\bar {f} (t) := \tilde f (\tau (t),t)$ satisfies $\cal B(\bar{ f}) \equiv \cal B (h_\rho)$ and 
$$
\Pi (\bar{f}(t)) \equiv \Pi (h _\rho).$$
Moreover,
$$
\partial _t \bar  {f} \big|_{t=0} = - \tau '(0) \vec n + \Phi_l.
$$
With $\partial _t \bar {f} \big|_{t=0} \in \ker D \Pi (h _\rho)$ we get $\tau '(0) = 0$, eg.
$$
\partial _t \bar{f} |_{t=0} =  \Phi_l.
$$
In Theorem \ref{thmstability1} we calculated
$$
\frac{d^2}{dt^2} \cal W (\bar {f}) \Big|_{t=0} < 0.
$$
In the case of the other instable directions $\Psi _l (u,v) = \cos (lv) \vec n(u,v)$ we can proceed in exactly the same way. 
This completes the proof. 
\end{proof}

\bibliography{bib}
\vspace{0.5cm}
\begin{center}
{\small
Ernst Kuwert, Mathematisches Institut, Universit\"at Freiburg, Eckerstra{\ss}e 1,
79104 Freiburg, Germany\\
{\sc Email: } ernst.kuwert@math.uni-freiburg.de\\
\vspace{0.5cm}
Johannes Lorenz, Mathematisches Institut, Universit\"at Freiburg, Eckerstra{\ss}e 1,
79104 Freiburg, Germany\\
{\sc Email: } johannes.lorenz@math.uni-freiburg.de \\
} 
\end{center}
\end{document}